\documentclass[12pt,notitlepage]{article}
\usepackage[margin=1in]{geometry}
\usepackage{times}
\usepackage{url}
\usepackage{hyperref}
  \usepackage{epsfig}
  \usepackage{amssymb}
  \usepackage{amsmath}
 \usepackage{amsthm}
  \usepackage{graphics}
  \usepackage{graphicx, latexsym, subfigure}
  \usepackage{enumitem}
  \setlist[description]{leftmargin=1cm,labelindent=1cm}

\usepackage[all,cmtip]{xy}

\usepackage{enumerate}
  
\newtheorem{theorem}{Theorem}[section]

\theoremstyle{definition}

\let\oldmarginpar\marginpar
\renewcommand\marginpar[1]{\-\oldmarginpar[\raggedleft\footnotesize #1]%
{\raggedright\footnotesize #1}}

\newcommand{\ZZ}{\mathbb{Z}}

\newcommand{\FF}{\mathbb{F}}

\DeclareMathOperator{\supp}{supp}
\DeclareMathOperator{\Out}{Out}

\title{Not all finitely generated groups have universal acylindrical actions}
\author{Carolyn R. Abbott}
\date{}

\begin{document}
\maketitle

\begin{abstract}
The class of acylindrically hyperbolic groups, which are groups that admit a certain type of non-elementary action on a hyperbolic space, contains many interesting groups such as non-exceptional mapping class groups and $\Out(\FF_n)$ for $n\geq 2$.  In such a group, a generalized loxodromic element is one that is loxodromic for some acylindrical action of the group on a hyperbolic space.  Osin asks whether every finitely generated group has an acylindrical action on a hyperbolic space for which all generalized loxodromic elements are loxodromic.  We answer this question in the negative, using Dunwoody's example of an inaccessible group as a counterexample.
\end{abstract}

\section{Introduction}

The action of a group $G$ on a metric space $X$ is \emph{acylindrical} if for all $\epsilon>0$ there exist constants $M,N\geq 0$ such that for all $x,y\in X$ with $d(x,y)\geq M$, the number of elements $g\in G$ satisfying $d(x,gx)\leq \epsilon$ and $d(y,gy)\leq \epsilon$ is at most $N$.  We say that a group is {\it acylindrically hyperbolic} if it admits a non-elementary acylindrical action on a hyperbolic space.

This definition is due to Osin, in \cite{O}, and it unifies several previous notions of groups which admit a non-elementary action on a hyperbolic space satisfying certain properties.  He shows that the class of groups satisfying properties such as Bestvina and Fujiwara's weak proper discontinuity (WPD) \cite{BF}, Hamenstadt's weak acylindricity \cite{H}, and Sisto's notion of the existence of weakly contracting elements \cite{S} are all equivalent and coincide with the class of acylindrically hyperbolic groups.

Acylindrically hyperbolic groups include a variety of interesting groups, such as non-exceptional mapping class groups, $\Out(\FF_n)$ for $n\geq 2$, right-angled Artin groups which are neither cyclic nor directly decomposable, many fundamental groups of compact 3-manifolds, 1-relator groups with at least three generators, non-elementary hyperbolic groups, and relatively hyperbolic groups.

Recall that given a group $G$ acting on a hyperbolic space $X$, an element $g\in G$ is called {\em elliptic} if the orbit $g\cdot s$ is bounded for some (equivalently any) $s\in X$, and it is called {\em loxodromic } if the map $\ZZ\to X$ defined by $n\mapsto g^ns$ is a quasi-isometric embedding for some (equivalently any) $s\in X$.  It is shown by Bowditch \cite{B} that every element of a group acting acylindrically on a hyperbolic space is  either elliptic or loxodromic.  However, an element of $G$ may be elliptic for some actions and loxodromic for others.  Consider, for example, the free group on two generators acting on its Cayley graph and acting on the Bass-Serre tree associated to the splitting $\FF_2\simeq \langle x\rangle *\langle y\rangle$.  In the former action, every non-trivial element is loxodromic, while in the latter action, all powers of $x$ and $y$ are elliptic.

An element is called \emph{generalized loxodromic} if it acts loxodromically for \emph{some} acylindrical action on a hyperbolic space.  We say a group $G$ has a {\it universal acylindrical action} if it has an acylindrical action on a hyperbolic space such that all generalized loxodromic elements act loxodromically.  In the case of the mapping class group, it is known that an element is generalized loxodromic if and only if it is pseudo-Anosov, in which case it acts loxodromically on the curve complex, so this acylindrical action is universal.  Non-elementary hyperbolic groups, subgroups of hyperbolic groups, and groups hyperbolic relative to a collection of subgroups, none of which are virtually cyclic or acylindrically hyperbolic, are known to admit universal acylindrical actions.  In addition, it follows from the work of Kim-Koberda in \cite{KK} that the acylindrical action of a right-angled Artin group on the extension graph is universal.

It should be noted that the question of universal acylindrical actions is trivial for non-acylindrically hyperbolic groups, as every group has an acylindrical action on a point, and this action is universal for non-acylindrically hyperbolic groups, which have no generalized loxodromic elements.  In addition, it is straight-forward to construct examples of infinitely generated acylindrically hyperbolic groups without universal acylindrical actions; Osin gives one example in \cite{O}.

%Hyperbolic groups and groups hyperbolic relative to a collection of subgroups, none of which are virtually cyclic or acylindrically hyperbolic, are also known to have universal acylindrical actions. \cite{O}

In \cite[Question 6.11]{O} Osin asks whether all finitely generated groups have a universal acylindrical action on a hyperbolic space.  We answer this question in the negative:

\begin{theorem}\label{thm:1} Dunwoody's inaccessible group does not admit a universal acylindrical action. \end{theorem}

{\bf Acknowledgements:} I would like to thank Denis Osin for useful comments on an earlier draft.

 \section{Dunwoody's inaccessible group}

Recall that a finitely generated group is accessible if the process of taking iterated splittings over finite subgroups terminates in finitely many steps.  Dunwoody \cite{D} constructs the following group as an example of a finitely generated group that is not accessible.  
 
Let $H$ be the group of permutations of the integers generated by the transposition $(0\,1)$ and the map $s(i)=i+1$, and let $H_i\subset H$ be the subgroup of permutations supported on $[-i,i]$, so that $H_i$ is the group of all permutations of $[-i,i]$.  Let $H_\omega=\cup_{i=0}^\infty H_i$ be the group of finitely supported permutations of $\ZZ$.  Let $V$ be the group of maps $\ZZ\to \ZZ/2\ZZ$, and let $V_i\subset V$ be the subgroup of maps supported on $[-i,i]$.  Then $H_i$ acts on $V_i$ via $^hv(j)=v(h^{-1}(j))$, so we can form the semidirect product $G_i'=V_i\rtimes H_i$, with the group operation $(v_0h_0)(v_1h_1)=(v_0{}^{h_0}v_1)(h_0h_1)$.   Let $z_i$ be the map that sends every element in $[-i,i]$ to $-1$ and all other elements to 1.  Then $z_i$ is central in $G_i'$, so we can form the direct product $K_i=\langle z_i\rangle \times H_i$.  Let $G_i$ be pairwise disjoint isomorphic copies of $G_i'$, and identify the subgroup $K_i$ with its image in $G_i$ and $G_{i+1}$.  We can then form the amalgamated product \[G_1*_{K_1} G_2*_{K_2}\cdots *_{K_n} G_{n+1}*_{K_{n+1}}\cdots.\]  Choose any $n$, collapse everything to the right of $K_n$, and call it $Q_n$, so that this product becomes \[G_1*_{K_1} G_2*_{K_2}\cdots *_{K_n}Q_n.\]  Notice that (in $G_{i+1}$) $K_i\cap K_{i+1}=H_i$, and since the $H_i$ are nested, $H_\omega\subset Q_n$.  We are now ready to define Dunwoody's group $J$ as the amalgamated free product \begin{align}J=G_1*_{K_1} G_2*_{K_2}\cdots *_{K_n} J_n,\end{align} where $J_n=Q_n*_{H_\omega}H$.  Notice that $J$ does not depend on the choice of $n$.

The group $J$ is finitely generated by $G_1$, which is finite, and $H$, which is finitely generated, but is not finitely presented.  It is shown by by Bowditch in \cite{B2} that given a spitting as in equation (1), $J$ is hyperbolic relative to the vertex groups appearing in that splitting, and thus $J$ is acylindrically hyperbolic.  As an amalgamated product over finite subgroups, it can be easily derived from \cite[Lemma 4.2]{MO} that the action on the Bass-Serre tree associated to any such splitting above is acylindrical.

The following theorem of Osin gives a necessary condition for there to be a universal acylindrical action.
\begin{theorem} [Osin \cite{O}] 
Suppose that a group $G$ acts acylindrically on a hyperbolic space. Then there exists a constant $N$ with the following property. Let $g$ be a loxodromic element. Then the centralizer $C_G(g)$ contains a cyclic subgroup of index less than $N$.
\end{theorem}

Thus to show that the group $J$ does not admit a universal acylindrical action, it suffices to find a sequence of elements $g_1,g_2,\dots$ in $J$ so that each is loxodromic for some acylindrical action on a hyperbolic space, and for any cyclic subgroup, $A_i\subseteq C_J(g_i)$, we have the index $|C_J(g_i):A_i|\to \infty$.

\begin{proof}[Proof of Theorem \ref{thm:1}]

Consider the decomposition $J=G_1*_{K_1} G_2*_{K_2}\cdots *_{K_i}G_{i+1}*_{K_{i+1}} J_{i+1}$.  Then $J$ acts acylindrically on the Bass-Serre tree associated to this splitting, and the product of any two elements in distinct vertex groups that are not contained in a common edge group is a loxodromic element for this action \cite{Se}. 

Let $\phi_i \in V_i\setminus\langle z_i\rangle$, $\psi_i\in V_{i+1}\setminus \langle z_{i}\rangle$, $t_i\in H_i$ and $s_i\in H_{i+1}$, so that $\phi_it_i\in G_i\setminus K_i$ and $\psi_is_i\in G_{i+1}\setminus K_i$.  Then $\phi_it_i\psi_is_i$ is a loxodromic element for the action on the Bass-Serre tree associated to the above splitting.   Let $p_i\in H_i\subset G_i\cap G_{i+1}$. In order for $p_i\in C_J(\phi_it_i\psi_is_i)$, we need $(p_i)(\phi_it_i\psi_is_i)=(\phi_it_i\psi_is_i)(p_i)$.  Simplifying each side yields 
\begin{align*} (p_i)(\phi_it_i\psi_is_i) &=[(p_i\phi_i)t_i](\psi_is_i)\\
&=[^{p_i}\phi_ip_it_i](\psi_is_i) \\
&= {}^{p_i}\phi_i[p_it_i\psi_is_i] \\
&=(^{p_i}\phi_i)\cdot\left(^{p_it_i}\psi_ip_it_is_i\right),
\end{align*}
where $^{p_i}\phi_i\in G_i$ and $^{p_it_i}\psi_ip_it_is_i\in G_{i+1}$.  Similarly, 
\begin{align*} (\phi_it_i\psi_is_i)(p_i)&=(\phi_it_i)(\psi_is_ip_i) \\
&=(\phi_i)(t_i\psi_is_ip_i) \\
&=(\phi_i)\cdot \left(^{t_i}\psi_it_is_ip_i\right),
\end{align*}
where $\phi_i\in G_i$ and $^{t_i}\psi_it_is_ip_i\in G_{i+1}$.
Now, $(^{p_i}\phi_i)\cdot\left(^{p_it_i}\psi_ip_it_is_i\right)=(\phi_i)\cdot \left(^{t_i}\psi_it_is_ip_i\right)$ if the following two conditions are satisfied:
\begin{description}
\item[(a)] $^{p_i}\phi_i=\phi_i$ as elements of $V_i\subset G_i$, i.e., as maps $\ZZ\to \ZZ/2\ZZ$ 
 \item[(b)] $^{p_it_i}\psi_ip_it_is_i={}^{t_i}\psi_it_is_ip_i$ as elements of $G_{i+1}$.  That is, $^{p_it_i}\psi_i={}^{t_i}\psi_i$ as elements of $V_{i+1}$ and $p_it_is_i=t_is_ip_i$ as elements of $H_{i+1}$, i.e., as permutations of $\ZZ$.
\end{description}
To satisfy (a), it suffices to have $\supp(p_i)\cap \supp(\phi_i)=\emptyset$.  To satisfy (b), it suffices to have $\supp(p_i)\cap \supp({}^{t_i}\psi_i)=\emptyset$ and $\supp(p_i)\cap \supp(t_is_i)=\emptyset$.

Now consider the elements 
\begin{align*}\phi_i(j)&=\begin{cases} -1 & j\in\{-i+1,-i+2\} \\ 1 & \textrm{else}\end{cases} \\ t_i&=(-i,\,-i+1)\\
\psi_i(j)&=\begin{cases} -1 & j\in\{-i,-i+1\} \\ 1 & \textrm{else}\end{cases} \\ s_i&=(-i-1,\,-i) 
\end{align*}  Then $\supp(\phi_i)=\{-i+1,-i+2\}$, $\supp(t_i)=\{-i,-i+1\}$, $\supp(\psi_i)=\{-i,-i+1\}$, and $\supp(s_i)=\{-i-1.-i\}$.  Notice that $^{t_i}\psi_i=\psi_i$.  Let $p_i\in H_i$ be any permutation with $\supp(p_i)\subseteq[-i+3,i]$.  Then (a) and (b) are satisfied, and so $p_i\in C_J(\phi_it_i\psi_is_i)$.  Thus the group of permutations of $[-i+3,i]$ is a finite subgroup of $C_J(\phi_it_i\psi_is_i)$.  But this subgroup has size $(2i-2)!$, which goes to infinity as $i$ does.  Therefore, for any cyclic subgroup, $A_i\subseteq C_J(\phi_it_i\psi_is_i)$, we have $|C_J(\phi_it_i\psi_is_i):A_i|>(2i-2)!$, which goes to infinity.

%Let $p_i\in H_i$ be any transposition with $\supp(p_i)\subseteq[-i+3,i]$.  Then (a) and (b) are satisfied, and so $p_i\in C_J(\phi_it_i\psi_is_i)$.  There are ${2i-2 \choose 2}=2i^2-5i+3$ choices for $p_i$, and $p_i\notin \langle \phi_it_i\psi_is_i\rangle$.  Thus, we have found a finite subgroup of $C_J(\phi_it_i\psi_is_i)$ whose size goes to infinity, and so for any cyclic subgroup, $A_i\subseteq C_J(\phi_it_i\psi_is_i)$, $|C_J(\phi_it_i\psi_is_i):A_i|>2i^2-5i+3$, which goes to infinity.
\end{proof}

Our method relies strongly on the fact that $J$ is inaccessible.  It is not known whether all finitely presented groups, which are necessarily accessible \cite{D2}, admit a universal acylindrical action.

 \end{document}